\documentclass[12pt]{amsart}
\usepackage{epsfig,color}

\usepackage{mathrsfs}
\usepackage{amssymb}

\theoremstyle{definition}

\newtheorem{theorem}{Theorem}[section]
\newtheorem{definition}[theorem]{Definition}

\newtheorem{proposition}[theorem]{Proposition}
\newtheorem{lemma}[theorem]{Lemma}
\newtheorem{claim}[theorem]{Claim}

\newtheorem{convention}[theorem]{Convention}
\newtheorem{corollary}[theorem]{Corollary}

\numberwithin{equation}{section}

\usepackage{enumerate}

\newcommand{\K}{\mathcal{K}}
\renewcommand{\L}{\mathcal{L}}

\newcommand{\eps}{\varepsilon}
\newcommand{\R}{\mathbb{R}}
\newcommand{\M}{\mathcal{M}}
\newcommand{\mcfK}{\mathcal{K}}
\newcommand{\mcfM}{\mathcal{M}}
\newcommand{\dilD}{\mathcal{D}}

\DeclareMathOperator*{\thi}{th}
\DeclareMathOperator*{\dist}{dist}
\DeclareMathOperator*{\graph}{graph}

\usepackage{comment}

\title{Flows with surgery revisited}

\author{Robert Haslhofer}

\begin{document}
\begin{abstract}In this paper, we introduce a new method to establish existence of geometric flows with surgery. In contrast to all prior constructions of flows with surgery in the literature our new approach does not require any a priori estimates in the smooth setting. Instead, our approach is based on a hybrid compactness theorem, which takes smooth limits near the surgery regions but weak limits in all other regions. For concreteness, here we develop our new method in the classical setting of mean-convex surfaces in $\mathbb{R}^3$, thus giving a new proof of the existence results due to Brendle-Huisken and Haslhofer-Kleiner. Other settings, including in particular free boundary surfaces, will be addressed in subsequent work.
\end{abstract}
\date{\today}
\maketitle

\tableofcontents

\section{Introduction}

Flows with surgery are a very powerful method to evolve geometric shapes, and have found many important applications in geometry and topology over the last 20 years.\\

For the Ricci flow, a surgery process has been first proposed in pioneering work by Hamilton \cite{Hamilton_survey,Hamilton_pic}. In his spectacular papers \cite{Per1,Per2}, Perelman successfully constructed a flow with surgery in three dimensions and used it to prove the Poincare and geometrization conjecture (see also \cite{KL_notes,MT,CaoZhu} for detailed expositions). This construction has been later generalized to other settings, including orbifolds \cite{DinkelbachLeeb,KL_orbifolds}, noncompact manifolds \cite{BBM1,BBM2}, and higher dimensions \cite{CZ,CTZ,Brendle_surgery1,Brendle_surgery2,Huang}, yielding in particular a classification of manifolds with positive isotropic curvature in dimension $4$ and in dimensions $n\geq 12$. Other fundamental geometric and topological applications include the work on positive scalar curvature metrics by Marques \cite{Marques} and Bamler-Kleiner \cite{BK_contr} and the proof of the generalized Smale conjecture by Bamler-Kleiner \cite{BK_Smale}.\\

For the mean curvature flow, a surgery construction was first given for two-convex hypersurfaces in $\mathbb{R}^N$ for $N\geq 4$ by Huisken-Sinestrari \cite{huisken-sinestrari3}, which in particular yielded a topological classification of two-convex hypersurfaces. The case $N=3$ has then been solved independently by Brendle-Huisken \cite{BrendleHuisken} and Kleiner and the author \cite{HK}.
The construction has been later generalized to other ambient manifolds \cite{BrendleHuisken_ambient,HKetover}, low entropy flows \cite{MramorWang,DanielsHolgate}, and higher codimensions \cite{Nguyen_surgery,LangfordNguyen}. Applications include the work on moduli-spaces in \cite{BHH1,BHH2}, the construction of foliations in \cite{HKetover,LiokumovichMaximo}, and the proof of the low entropy Sch\"onflies conjecture in \cite{DanielsHolgate}.\\

All the above existence proofs rely on a strong collection of a priori estimates, most notably noncollapsing and convexity estimates, which are only available in a rather restricted set of situations already exploited above. The purpose of the present paper is to introduce a new method to establish existence of geometric flows with surgery, which does not rely on any a priori estimates in the smooth setting. To develop the method we give a new proof of the following theorem due to Brendle-Huisken \cite{BrendleHuisken} and Kleiner and the author \cite{HK}:

\begin{theorem}[existence of mean curvature flow with surgery]
Mean curvature flow with surgery exists starting at any closed mean-convex embedded surface in $\mathbb{R}^3$.
\end{theorem}

Moreover, our new method also enables the construction of flows with surgery in situations that have been inaccessible with prior techniques. The most direct instance of this is the construction of free boundary mean curvature flow, which will be given in a forthcoming paper \cite{H_fb_surgery}. We recall that Andrews' noncollapsing proof for closed surfaces \cite{andrews1} does not seem to generalize in any straightforward way to the case with boundary, so the new method introduced in the present paper is indeed crucial to address the free boundary setting. Another related open problem concerns the construction of a flow with surgery for mean-convex hypersurfaces in $\mathbb{R}^4$, where singularities have been recently classified in our joint work with various collaborators in \cite{CHH_wings,CHH_translators,DH_shape,DH_norotation,CDDHS,CHH_classification}.\\

Let us now outline our new method. For the purpose of this paper a flow with surgery is a $(\delta,\mathcal{H})$-flow (see Definition \ref{def_MCF_surgery}). Here, $\delta>0$ captures the quality of the surgery necks and $\mathcal{H}$ is a triple of curvature scales $H_{\textrm{trigger}}\gg H_{\textrm{neck}}\gg H_{\textrm{thick}}\gg 1$, which is used to specify more precisely when and how surgeries are performed. To begin with, comparing a $(\delta,\mathcal{H})$-flow $\mathcal{K}$ and the level set flow $\L$ (see \cite{CGG,evans-spruck}) with the same initial condition $K_0$, we prove that
\begin{equation}\label{dis_level_set}
d_{\mathrm{H}}(\K,\L)\leq C H_{\mathrm{neck}}^{-1},
\end{equation}
where $d_{\mathrm{H}}$ denotes the Hausdorff distance of the space-time tracks. This can be viewed as quantitative version of the main result of Head \cite{Head} and Lauer \cite{Lauer}.

We then study sequences $\K^j$ of $(\delta,\mathcal{H}^j)$-flows, with the same mean-convex initial condition $K_0$, where $\delta\leq\bar{\delta}$ and where the curvature scales $\mathcal{H}^j$ improve along the sequence, namely
\begin{equation}
\min\left(H^j_{\textrm{thick}} , \frac{  H^j_{\textrm{neck}}}{H^j_{\textrm{thick}} }, \frac{  H^j_{\textrm{trigger}}}{H^j_{\textrm{neck}} }   \right)\to \infty.
\end{equation}
Given any sequence of rescaling factors $\lambda_j\to \infty$, we consider the blowup sequence
\begin{equation}
\widetilde{\K}^j:=\mathcal{D}_{\lambda_j}( \K^j-X_j),
\end{equation}
which is obtained from $\mathcal{K}^j$ by translating $X_j$ to the origin and parabolically rescaling by $\lambda_j$. To keep track of multiplicities we also consider the associated family of Radon measures
\begin{equation}\label{ass_radon_intro}
\widetilde{\mathcal{M}}^{j}=\big\{  \tilde{\mu}^j_t=\mathcal{H}^2 \lfloor \partial \tilde{K}^j_t\big\},
\end{equation}
where $K_t:=K_t^-$ at surgery times. After passing to a subsequence we can assume that 
\begin{equation}
\Lambda:=\lim_{j\to \infty} \frac{\lambda_j}{H_{\textrm{neck}}^j} \in [0,\infty]
\end{equation}
exists. We then call $(\widetilde{\mathcal{K}}^j,\widetilde{\mathcal{M}}^j)$ a \emph{$\Lambda$-blowup sequence}. These sequences are our main object of study, since their analysis is the key towards understanding the structure of singularities and of high curvature regions. Here, in the three scenarios $\Lambda=0$, $0<\Lambda<\infty$, and $\Lambda=\infty$ the rescaled neck radius goes to zero, stays finite, or goes to infinity, respectively.

If $\Lambda=0$, then using \eqref{dis_level_set} we show that $\widetilde{\K}^j$ Hausdorff converges to a blowup limit of the level-set flow. Moreover, using Huisken's monotonicity formula \cite{Huisken_monotonicity} we exclude microscopic surgeries, i.e. we show that if a $0$-blowup sequence Hausdorff converges to a multiplicity-one plane, then any parabolic ball $P(0,R)$ contains no points modified by surgeries for $j$ large enough. Since blowup limits of the level-set flow are already understood thanks to the theory of White \cite{White_size,White_nature}, we can from now on assume that $\Lambda>0$.

We then prove a hybrid compactness theorem, which says that any $\Lambda$-blowup sequence $(\widetilde{\mathcal{K}}^j,\widetilde{\mathcal{M}}^j)$ subsequentially converges to a limit $(\mathcal{K},\mathcal{M})$. Here, in the case $\Lambda\in (0,\infty)$ the limit is what we call an \emph{ancient Brakke $\delta$-flow}, i.e. loosely speaking an ancient integral Brakke flow modified by surgeries on $\delta$-necks. The convergence is in the smooth sense near the surgery regions but in the weak sense of Brakke flows in all other regions. In the case $\Lambda=\infty$ the limit is either a Brakke flow without surgeries or a (quasi)static multiplicity-one plane.

We then analyze all possible limits $(\mathcal{K},\mathcal{M})$ provided by our hybrid compactness theorem, which we call \emph{generalized limit flows}. Adapting the arguments from White's work \cite{White_size,White_nature} to our setting of flows with surgery we show that all such generalized limit flows have muliplicity-one, are smooth away from a singular set of parabolic Hausdorff dimension at most 1, and have nonnegative second fundamental form at their regular part.

Combining the above results, and taking also into account the recent classification results from \cite{BC,ADS}, we establish a canonical neighborhood theorem. Specifically, we show that for any $\Lambda$-blowup sequence and any sequence of space-time points $X_j\in\partial\K^j$ with $H(X_j)\to\infty$, the parabolically rescaled flows
$\mathcal{D}_{H(X_j)}( \K^j-X_j)$ subsequentially converge to either (a) the evolution of a standard cap preceded by a round shrinking cylinder, or (b) a round shrinking cylinder, round shrinking sphere, translating bowl or ancient oval.

Finally, we use the canonical neighborhood theorem, via a continuity argument similarly as in \cite{HK}, to establish existence of mean curvature flow with surgery.

\bigskip

\noindent\textbf{Acknowledgments.}
I thank Salim Deaibes and Jonathan Zhu for helpful discussions. This research has been supported by an NSERC Discovery Grant and a Sloan Research Fellowship.
\bigskip

\section{Definitions and basic properties}

In this section, we define flows with surgery and establish some basic properties. The definitions are variants of the ones from the work by Kleiner and the author \cite{HK}, however, with the fundamental difference that we now do not assume $\alpha$-noncollapsing a priori.\\

Let us begin with the following flexible notion of $\delta$-flows, which describes mean-convex flows with finitely many (possible none) necks replaced by standard caps:

\begin{definition}[$\delta$-flow]\label{def_alphadelta}
A \emph{$\delta$-flow} $\K$ is a collection of finitely many smooth mean-convex mean curvature flows $\{K_t^i\subseteq \R^3\}_{t\in[t_{i-1},t_{i}]}$ ($i=1,\ldots,k$; $t_0<\ldots< t_k$) such that:
\begin{enumerate}[(a)]
\item for each $i=1,\ldots,k-1$, the final time slices of some collection of disjoint strong $\delta$-necks (see Definition \ref{def_strongneck}) are replaced by pairs of standard caps (see Definition \ref{def_replacestd}),
 giving a domain $K^\sharp_{t_{i}}\subseteq K^{i}_{t_{i}}=:K^-_{t_{i}}$.\label{def_delta_1st}\label{item_delta1}
\item the initial time slice of the next flow, $K^{i+1}_{t_{i}}=:K^+_{t_{i}}$, is obtained from $K^\sharp_{t_{i}}$ by discarding some connected components.
\item there exists $r_\sharp=r_\sharp(\K)>0$, such that all necks in item \eqref{item_delta1} have radius $r\in[\frac{1}{2}r_\sharp,2 r_\sharp]$.
\end{enumerate}
\end{definition}

The above definition relies on the following two further definitions:

\begin{definition}[strong $\delta$-neck]\label{def_strongneck}
We say that a $\delta$-flow $\K=\{K_t\}_{t\in I}$ has a \emph{strong $\delta$-neck} with center $p$ and radius $r$ at time $t_0\in I$, if
$\{r^{-1}\cdot(K_{t_0+r^2t}-p)\}_{t\in(-1,0]}$ is $\delta$-close in $C^{\lfloor 1/\delta\rfloor}$ in $B_{1/\delta}(0)\times (-1,0]$ to the evolution of a solid round cylinder ${D}^{2}\times \R$ with radius $1$ at $t=0$.
\end{definition}

\begin{definition}[replacing a $\delta$-neck by standard caps]\label{def_replacestd}
We say that the final time slice of a strong $\delta$-neck ($\delta\leq\tfrac{1}{100\Gamma}$) with center $p$ and radius $r$ is \emph{replaced by a pair of standard caps},
if the presurgery domain $K^-$ is replaced by a postsurgery domain $K^\sharp\subseteq K^-$ such that:
\begin{enumerate}[(a)]
\item the modification takes places inside a ball $B=B(p,5\Gamma r)$.
 \item there are bounds for the second fundamental form and its derivatives:
$$\sup_{\partial K^\sharp\cap B}|{\nabla^\ell A}|\leq C_\ell r^{-1-\ell}\qquad (\ell=0,1,2,\ldots).$$
 \item for every point $p_\sharp\in \partial K^\sharp\cap B$ with $\lambda_1(p_\sharp)< 0$, there is a point
 $p_{-}\in\partial K^{-}\cap B$ with $$\frac{\lambda_1}{H}(p_{-})\leq\frac{\lambda_1}{H}(p_{\sharp}).$$
 \item the domain $r^{-1}\cdot(K^\sharp-p)$ is $\delta'(\delta)$-close in $B(0,10\Gamma)$ to a pair of disjoint standard caps,
that are at distance $\Gamma$ from the origin. Here, $\delta'(\delta)\to 0$ as $\delta\to 0$.\label{def_surgery_delta_hat}
\end{enumerate}
\end{definition}

Here, a \emph{standard cap} is a smooth convex domain $K^{\textrm{st}}\subset \R^3$ that coincides with a solid round half-cylinder of radius $1$ outside a ball of radius $10$.
Also, we will often use the phrase that an open set $U$ contains \emph{points modified by surgery} at time $t$ if $(K_t^-\setminus K_t^\sharp)\cap U\neq \emptyset$.

\bigskip

In this paper, a mean curvature flow with surgery is a $(\delta,\mathcal{H})$-flow defined as follows:

\begin{definition}[flow with surgery]\label{def_MCF_surgery}
A \emph{$(\delta,\mathcal{H})$-flow}, where $\mathcal{H}=(H_{\textrm{thick}},H_{\textrm{neck}},H_{\textrm{trigger}})$, is a $\delta$-flow $\{K_t\subset \R^3\}_{t\geq 0}$ with smooth compact mean-convex initial condition $K_0\subset \R^3$ such that:
\begin{enumerate}[(a)]
\item $H\leq H_{\textrm{trigger}}$ everywhere, and
surgery and/or discarding occurs precisely at times $t$ when $H=H_{\textrm{trigger}}$ somewhere. 
\item The collection of necks in Definition \ref{def_alphadelta}\eqref{def_delta_1st} is a minimal collection of solid $\delta$-necks of curvature $H_{\textrm{neck}}$ which
separate the set $\{H=H_{\textrm{trigger}}\}$ from $\{H\leq H_{\textrm{thick}}\}$ in $K_t^-$.\label{def_mcf_surgery2}
\item $K_t^+$ is obtained from $K_t^\sharp$ by discarding precisely those connected components with $H>\tfrac{1}{10}H_{\textrm{neck}}$ everywhere. For each pair of facing surgery caps, precisely one is discarded. 
\item If a strong $\delta$-neck from item \eqref{def_mcf_surgery2} also is a strong $\hat{\delta}$-neck for some $\hat{\delta}<\delta$, then property \eqref{def_surgery_delta_hat} of Definition \ref{def_replacestd} also holds with $\hat{\delta}$ instead of $\delta$.
\end{enumerate}
\end{definition}

We remark that the third item differs from \cite[Definition 1.17]{HK}, where it was only assumed that $H>H_{\textrm{thick}}$ on the discarded components. Here, we achieve the better inequality $H>\tfrac{1}{10}H_{\textrm{neck}}$ by performing surgeries on `innermost' necks, see the existence proof.\\

As a direct consequence of the definitions we have the following basic properties:

\begin{proposition}[basic properties]\label{basic_prop}
There exist $\bar{\delta}>0$ and $\Gamma_0<\infty$, such that any $\delta$-flow $\mathcal{K}$ with surgery quality $\delta\leq\bar{\delta}$ and cap separation parameter $\Gamma\geq \Gamma_0$ satisfies the following:
\begin{enumerate}[(a)]
\item If $p$ is the center of a surgery neck of radius $r$, then there are no other surgeries in $B(p,\frac{1}{10}\delta^{-1} r)$.
\item For every ball ${B}$ we have $|{\partial K_{t_1}\cap {B}}|\leq |{\partial K'\cap {B}}|$
for every $K'$ that agrees with $K_{t_1}$ outside ${B}$ and satisfies $K_{t_1}\subseteq K'\subseteq K_{t_0}$ for some $t_0<t_1$.
\end{enumerate}
Moreover,  any $(\delta,\mathcal{H})$-flow satisfies $H>C^{-1}$ and $|A|\leq CH$, where $C=C(K_0)<\infty$.
\end{proposition}

\begin{proof}
The spatial separation of surgeries follows directly from the definitions. Next, the one-sided minimization follows using the geometric measure theory argument in \cite[Section 3]{White_size} or \cite[Section 5]{Head}; see also \cite[Appendix A]{HK}. Finally, the claimed bounds for $H$ and $|A|/H$ follow from the maximum principle and the definition of surgeries.
\end{proof}

\begin{convention}\label{conv}
We now fix a suitable standard cap $K^{\textrm{st}}$ and cap separation parameter $\Gamma<\infty$. Moreover, throughout this paper $\bar{\delta}>0$ and $\eps>0$ denote sufficiently small constants (by convention, these constants can be decreased finitely many times as needed or convenient).
\end{convention}

\bigskip

\section{Distance to level set flow} 

It has been proved by Head \cite{Head} and Lauer \cite{Lauer} that for suitably degenerating surgery parameters the flow with surgery converges to the level set flow from \cite{CGG,evans-spruck}. In this section, we improve this to a more quantitative estimate. We will identify $\delta$-flows with their spacetime track
\begin{equation}
\mathcal{K}=\bigcup_t K_t \times\{t\} \subset \mathbb{R}^3\times \mathbb{R}. 
\end{equation}
Here, we set $K_{t_i}:=K_{t_i}^-$ at surgery times to ensure that $\mathcal{K}$ is a closed subset of space-time. Note that $K_{t_i}^+$ can be recovered from $\mathcal{K}$ by considering the limit of time slices as $t\searrow t_i$.

\begin{proposition}[distance to level set flow]\label{prop_levelset}
There exist constants $\bar{\delta}>0$ and $C=C(K_0)<\infty$ such that if $\K$ is a 
$(\delta,\mathcal{H})$-flow ($\delta\leq\bar{\delta}$) with initial condition $K_0$, where $H_{\textrm{neck}}\geq C$, and $\L$ is the level set flow with the same initial condition $K_0$, then
\begin{equation}
d_{\mathrm{H}}(\K,\L)\leq C H_{\mathrm{neck}}^{-1},
\end{equation}
where $d_{\mathrm{H}}$ denotes the Hausdorff distance of the spacetime tracks.
\end{proposition}

\begin{proof}
	Since mean curvature flow with surgery satisfies the avoidance principle, we have
	\begin{equation}
	\K\subseteq \L.
	\end{equation}
	To estimate the distance from the other direction, we start with the following claim:
	
	\begin{claim}[containment]\label{claim_thick}
	 There exists $\eta<\infty$, such that if $t$ is a surgery time of $\mathcal{K}$ and $B(p,r) \subseteq K^{-}_t$ is a ball of radius $r> \eta H_{\mathrm{neck}}^{-1}$, then $B(p,r) \subseteq K^{+}_t$.
	\end{claim}
	\begin{proof}[Proof of Claim \ref{claim_thick}]
	Since $B(p,r) \subseteq K^{-}_t$ clearly does not fit into a thin long neck, we must have $B(p,r) \subseteq K^{\sharp}_t$, provided $\bar{\delta}$ is small enough and $\eta$ is large enough.
	If $B(p,r)$ was contained in a discarded component, then we could find a boundary point with $H\leq n/r$. However, by Definition \ref{def_MCF_surgery} (flow with surgery) the discarded components satisfy $H> \tfrac{1}{10} H_{\mathrm{neck}}$ everywhere. Hence, we conclude that $B(p,r) \subseteq K^{+}_t$ provided $\eta$ is large enough.
	\end{proof}

Next, observe that there is a $C_0=C_0(K_0)<\infty$ such that for all $\gamma>0$ small enough, we have $C_0^{-1}\leq H\leq C_0$ for $t\leq C_0\gamma$. In particular, there is some $\tau \in [C_0^{-1}\gamma, C_0 \gamma]$ such that \begin{equation}
d(K_\tau,\mathbb{R}^3\setminus K_0)=\gamma.
\end{equation}
We choose $\gamma=\eta/H_{\textrm{neck}}$, which is allowed provided $H_{\textrm{neck}}$ is large enough. With the corresponding $\tau$, we let $\mathcal{L}^\tau$ be the level set flow with initial condition $L^\tau_0=K_\tau$. We then have the inclusion $ L^\tau_t \subseteq K_t$ for all $t$ with the estimate
\begin{equation}
d( L^\tau_t,\mathbb{R}^3\setminus K_t)\geq \gamma.
\end{equation}
Indeed, the distance equals $\gamma$ at $t=0$ and is nondecreasing away from the surgery times, and if $t_i$ is a surgery time, then  $d(L^\tau_{t_i}, \mathbb{R}^3\setminus K_{t_i}^-) \geq \gamma$ 
    together with Claim \ref{claim_thick} (containment) implies $d(L^\tau_{t_i}, \mathbb{R}^3\setminus K_{t_i}^+) \geq \gamma$.
Summarizing, we have shown that
 \begin{align}
  \mathcal{L}^\tau \subseteq \mathcal{K} \subseteq \mathcal{L},
 \end{align}
 where $\tau$ is comparable to $H_{\mathrm{neck}}^{-1}$. Since $\mathcal{L}^\tau$ and $\mathcal{L}$ just differ by a shift by $\tau$ in time direction, this implies the assertion.
\end{proof}

\bigskip

\section{Hybrid compactness theorem}

Let  $\K^j$ be a sequence of $(\delta,\mathcal{H}^j)$-flows, where $\delta\leq\bar{\delta}$, with the same mean-convex initial condition $K_0$, and suppose that $\mathcal{H}^j\to\infty$, 
where we use the abbreviation
\begin{equation}
\mathcal{H}^j\to \infty \quad:\Leftrightarrow\quad
\min\left(H^j_{\textrm{thick}} , \frac{  H^j_{\textrm{neck}}}{H^j_{\textrm{thick}} }, \frac{  H^j_{\textrm{trigger}}}{H^j_{\textrm{neck}} }   \right)\to \infty.
\end{equation}
Let $\lambda_j$ be a sequence of rescaling factors satisfying $\lambda_j\to \infty$. Consider the blowup sequence
\begin{equation}
\widetilde{\K}^j:=\mathcal{D}_{\lambda_j}( \K^j-X_j),
\end{equation}
which is obtained from $\mathcal{K}^j$ by translating $X_j$ to the origin and parabolically rescaling by $\lambda_j$. Moreover, we also consider the associated family of Radon measures
\begin{equation}\label{ass_radon}
\widetilde{\mathcal{M}}^{j}=\big\{  \tilde{\mu}^j_t=\mathcal{H}^2 \lfloor \partial \tilde{K}^j_t\big\},
\end{equation}
where we set $K_t:=K_t^-$ at surgery times. Note that $\widetilde{\mathcal{M}}^{j}$ satisfies Brakke's inequality for test functions supported away from the surgery regions. In this section, we will establish how to pass to a subsequential limit  of $(\widetilde{\K}^j,\widetilde{\M}^j)$, as a hybrid of smooth limits near the surgeries and Brakke flow limits elsewhere. We will refer to these limiting objects as generalized limit flows.\\

To begin with, after passing to a subsequence we can assume that 
\begin{equation}
\Lambda:=\lim_{j\to \infty} \frac{\lambda_j}{H_{\textrm{neck}}^j} \in [0,\infty]
\end{equation}
exists. We then say that $(\widetilde{\mathcal{K}}^j,\widetilde{\mathcal{M}}^j)$ is a \emph{$\Lambda$-blowup sequence}. In the three scenarios $\Lambda=0$, $0<\Lambda<\infty$, and $\Lambda=\infty$ the rescaled neck radius goes to zero, stays finite, or goes to infinity, respectively. Let us first analyze the scenario $\Lambda=0$.

\begin{proposition}[convergence to level-set flow blowup]\label{prop_L0}
 If $\Lambda=0$, then a subsequence of $\widetilde{\mathcal{K}}^j$ Hausdorff converges to a limit $\mathcal{K}$, which is a blowup limit of the level set flow of $K_0$. 
\end{proposition}

\begin{proof}Let $\mathcal{L}$ be the level set flow of $K_0$. By Proposition \ref{prop_levelset} (distance from level set flow) we have  $d_{\mathrm{H}}(\mathcal{K}^j,\mathcal{L}) \leq C/H^j_{\text{neck}}$, hence
\begin{align}
d_{\mathrm{H}}(\widetilde{\mathcal{K}}^j,\mathcal{D}_{\lambda_j}(\mathcal{L}-X_j)) \leq C\lambda_j/H^j_{\text{neck}}.
\end{align}
Since $\lim_{j \to \infty}\lambda_j/H^j_{\text{neck}}=0$ by assumption, it follows that $\widetilde{\mathcal{K}}^j$ and $\mathcal{D}_{\lambda_j}(\mathcal{L}-X_j)$ converge subsequentially in the Hausdorff sense to the same limit $\K$. This proves the proposition.
\end{proof}

Let us now analyze the main scenario $0<\Lambda<\infty$. As mentioned above, the limits we take will be a hybrid of smooth limits near the surgeries and Brakke flow limits elsewhere. Let us now carefully define the limiting objects and the notion of convergence:

\begin{definition}[{ancient Brakke $\delta$-flow}]\label{Def_Brakke_surgery}
An \emph{ancient Brakke $\delta$-flow} is a pair $(\mathcal{K},\mathcal{M})$ consisting of a nested family of closed sets $\mathcal{K}=\{K_t\}_{t \in (-\infty,T)}$ and a family of Radon measures $\mathcal{M}=\{\mu_t\}_{t \in (-\infty,T)}$, for which
there exists a constant $r_\sharp=r_\sharp(\K,\M)\in (0,\infty)$ and a disjoint collection of two-sided parabolic balls $P_i=B(p_i,50\Gamma r_{\sharp}) \times (t_i-\tfrac12 r_{\sharp}^2,t_i+\eps r_{\sharp}^2)$ such that:
\begin{enumerate}[(a)]
 \item For $t\in (t_i-\tfrac12 r_{\sharp}^2,t_i+\eps r_{\sharp}^2)$, we have $\mu_t\lfloor B(p_i,50\Gamma r_{\sharp})=\mathcal{H}^2 \lfloor \partial K_t\cap B(p_i,50\Gamma r_{\sharp})$, and for $t\neq t_i$ the sets $K_t\cap B(p_i,50\Gamma r_{\sharp})$ are smooth and evolve by mean curvature flow. At time $t_i$ a strong $\delta$-neck (see Definition \ref{def_strongneck}) of radius $r_i\in [r_\sharp/2,2r_\sharp]$ centered at $(p_i,t_i)$ is replaced by a pair of standard caps (see Definition \ref{def_replacestd}), and possibly some connected components of $K_{t_i}$ are discarded.\label{Def_Brakke_surgery1}
 \item Considering the somewhat smaller $P_i'=B(p_i,25\Gamma r_{\sharp}) \times (t_i-\tfrac14 r_{\sharp}^2,t_i+\tfrac12\eps r_{\sharp}^2)$, we have that $\mathcal{M}$ is an integral Brakke flow away from $\cup_i P_i'$, and $ \partial \mathcal{K}\setminus \cup_i P_i'=\textrm{spt}\mathcal{M}\setminus \cup_i P_i' $.\footnote{Recall that the support of an integral Brakke flow consists of all space-time points with Gaussian density at least $1$. For a detailed introduction to integral Brakke flows see e.g. Ilmanen's monograph \cite{Ilmanen}.}
\end{enumerate}
\end{definition}

\begin{definition}[{convergence to Brakke $\delta$-flow}]\label{Def_brakke_conv_surg}
A blowup sequence $(\widetilde{\K}^j,\widetilde{\M}^j)$  \emph{converges to a Brakke $\delta$-flow $(\mathcal{K},\mathcal{M})$} if the following holds:
 \begin{enumerate}[(a)]
 \item The space-time tracks $\widetilde{\mathcal{K}}^j$ Hausdorff converge to the space-time track $\mathcal{K}$.
  \item If $P_i=B(p_i,50\Gamma r_{\sharp}) \times (t_i-\tfrac12 r_{\sharp}^2,t_i+\eps r_{\sharp}^2)$ is a surgery region of $(\mathcal{K},\mathcal{M})$ as in Definition \ref{Def_Brakke_surgery} (ancient Brakke $\delta$-flow), then for some $(p_i^j,t_i^j)\in \widetilde{\mathcal{K}}^j$ converging to $(p_i,t_i)$ the forwards and backwards portion $\{\tilde{K}^j_{t+t_i^j}-p^j_i\}_{t \in (0,\eps r_{\sharp}^2)}$ and $\{\tilde{K}^j_{t+t_i^j}-p^j_i\}_{t \in (-\frac12 r_{\sharp}^2,0]}$ converge smoothly to $\{K_{t+t_i}-p_i\}_{t \in (0,\eps r_{\sharp}^2)}$ and $\{K_{t+t_i}-p_i\}_{t \in (-\frac12 r_{\sharp}^2,0]}$, respectively, in $B(0,50\Gamma r_{\sharp})$.\label{Def_brakke_conv_surg2}
      \item $\widetilde{\M}^j$ converges to ${\M}$ in the sense of Brakke flows away from $\cup_i P_i'$.
 \end{enumerate}
\end{definition}

\begin{theorem}[hybrid compactness]\label{prop_compactness}
Any blowup sequence $(\widetilde{\K}^j,\widetilde{\M}^j)$ with $\Lambda \in (0,\infty)$ as above has a subsequence that converges to an ancient Brakke $\delta$-flow $(\K,\M)$.
\end{theorem}

\begin{proof}
To begin with, by passing to a subsequence, we may assume that $\widetilde{\mathcal{K}}^j$ Hausdorff converges to a limit $\mathcal{K}=\{K_t\}_{t \in (-\infty,T)}$, which is a nested family of closed sets.\\
Now, for any $R<\infty$, consider the two-sided parabolic ball $P(0,R)=B(0,R)\times (-R^2,R^2)$ centered at the origin. Since $\Lambda>0$, the number $N_R^j$ of surgery centers of  $\widetilde{\mathcal{K}}^j$ in $P(0,R)$ is uniformly bounded. After passing to a subsequence, we can assume that $N_R^j=N_R$ is independent of $j$. Let $\{(p^j_i,t^j_i)\}_{i=1}^{N_R}$ be the surgery centers of  $\widetilde{\mathcal{K}}^j$ in $P(0,R)$. Note that the surgeries are performed on necks of radius $r^j_i \in [\Lambda_j/2,2\Lambda_j]$, where $\Lambda_j := \lambda_j/H^j_{\text{neck}}$. By passing to a further subsequence, we can assume that
\begin{align}\label{limit_data}
(p^j_i,t^j_i) \to (p_i,t_i)\qquad \textrm{and} \qquad
 r^j_i \to r_i \in [\Lambda/2,2\Lambda].
\end{align}
Set $r_\sharp:=\Lambda$, and recall that $\eps>0$ is a small fixed constant by Convention \ref{conv}.

\begin{claim}[curvature bounds]\label{curv_bounds}
For each $P_i=B(p_i,50\Gamma r_{\sharp}) \times (t_i-\tfrac12 r_{\sharp}^2,t_i+\eps r_{\sharp}^2)$ we have
\begin{equation}
\limsup_{j\to\infty} \sup_{ \partial \widetilde{\mathcal{K}}^j\cap P_i}|\nabla^\ell A|<\infty\qquad (\ell=0,1,2,\ldots).
\end{equation}
\end{claim}
\begin{proof}[{Proof of Claim \ref{curv_bounds}}]
Since $\widetilde{\mathcal{K}}^j$ has a strong $\delta$-neck centred at $(p^j_i,t^j_i)$ of radius $r^j_i$ (see Definition \ref{def_strongneck}), $\widetilde{\mathcal{K}}^j$ must have $|\nabla^\ell A|$ comparable to $(r_i^j)^{-1-\ell}$ in $B(p^j_i,\tfrac12 \delta^{-1}r^j_i) \times (t^j_i-\tfrac12 (r^j_i)^2,t^j_i)$. At the surgery time $t=t^j_i$, since by Definition \ref{def_replacestd} (replacing a $\delta$-neck by standard caps) the modification takes place in a ball of radius $5\Gamma r_i^j$ with scale invariant curvature bounds, we also obtain that $|\nabla^\ell A|$ is bounded by a controlled multiple of $(r_i^j)^{-1-\ell}$ in $B(p^j_i,\tfrac12 \delta^{-1}r^j_i)$. Finally, by the local regularity theorem \cite{white_regularity} we get curvature bounds for $t\in (t^j_i,t^j_i+5\varepsilon (r^j_i)^2)$. Remembering also \eqref{limit_data} and $\delta\leq \tfrac{1}{100\Gamma}$, this proves the claim.
\end{proof}

By Claim \ref{curv_bounds} (curvature bounds) the convergence in the two-sided parabolic balls $P_i$ is smooth in the sense of Definition \ref{Def_brakke_conv_surg}(\ref{Def_brakke_conv_surg2}). Let us also consider the somewhat smaller $P_i'=B(p_i,25\Gamma r_{\sharp}) \times (t_i-\tfrac14 r_{\sharp}^2,t_i+\tfrac12\eps r_{\sharp}^2)$. Away from $\cup_i P_i'$ we can pass to a subsequential limit in the sense of integral Brakke flows \cite{Ilmanen}. In the surgery regions we define our measure by declaring that $\mu_{t}(A):=\mathcal{H}^2( \partial K_{t} \cap A)$ for $A\subseteq B(p_i,50\Gamma r_{\sharp})$ and $t\in (t_i-\tfrac12 r_{\sharp}^2,t_i+\eps r_{\sharp}^2)$. Note that this way of defining the measure is consistent in the overlap regions.\\
We now consider a sequence $R_j \to \infty$, and pass to a diagonal subsequence of the above construction, to obtain a global limit $(\mathcal{K},\mathcal{M})$. Observe that our limit satisfies all the properties listed in Definition \ref{Def_Brakke_surgery}\eqref{Def_Brakke_surgery1}, and that $\mathcal{M}$ is an integral Brakke flow away from $\cup_i P_i'$.\\

To prove that the support of $\mathcal{M}$ agrees with $\partial \mathcal{K}$ away from $\cup_i P_i'$, we will argue similarly as in \cite[Section 5]{White_size}.
Since $\widetilde{\mathcal{K}}^j  \to \mathcal{K} $ in the Hausdorff sense, for any $X \in \partial \mathcal{K} \setminus \cup_i P_i'$ there exist points $X_j \in \partial \widetilde{\mathcal{K}}^j = \mathrm{spt} \widetilde{\mathcal{M}}^j $ such that $X_j \to X$, which by upper semicontinuity of the Gaussian density implies $X \in \mathrm{spt} \mathcal{M}$.
To address the other inclusion, recall that by the general stratification results \cite[Section 9]{White_stratification} almost every $X \in \mathrm{spt} \mathcal{M} \setminus \cup_i P_i'$ admits a tangent flow $(\mathcal{K}',\mathcal{M}')$ which is a static or quasistatic plane, possibly with multiplicity. Hence, there exists a plane $P$ and a time $T\leq \infty$ such that
 \begin{align}
  \mathrm{spt} \mathcal{M}' = P \times \overline{(-\infty,T)}.
 \end{align}
It follows that either $\mathcal{K}'=H \times \overline{(-\infty,T)}$ for some halfspace $H \subset \R^3$ bounded by $P$, or $\mathcal{K}' = P \times \overline{(-\infty,T)}$.
Hence, there exist points in the complement of $\mathcal{K}$ that are arbitrarily close to $X$. Namely, $X \in \partial \mathcal{K}$. This shows that a dense subset of $\mathrm{spt} \mathcal{M} \setminus \cup_i P_i'$ is contained in $\partial \mathcal{K} \setminus \cup_i P_i'$. We conclude that $\mathrm{spt} \mathcal{M} \setminus \cup_i P_i'=\partial \mathcal{K} \setminus \cup_i P_i'$. This finishes the proof of the theorem.
\end{proof}

Finally, let us deal with the scenario $\Lambda=\infty$:

\begin{proposition}[large blowups]\label{prop_LI}
 If $\Lambda=\infty$, then $(\widetilde{\K}^j,\widetilde{\M}^j)$ subsequentially converges either (a) in the Hausdorff and Brakke sense to a weakly mean-convex Brakke flow without surgeries or (b) smoothly to a static or backwards/forwards quasistatic multiplicity-one plane.\footnote{In the (usual) backwards quasistatic case $K_t$ is a halfspace for $t\leq T$ and empty for $t>T$. In the (hereby newly defined) forwards quasistatic case $K_t$ is entire space for $t\leq T$ and a halfspace for $t> T$.}
\end{proposition}

\begin{proof}
If for all $R<\infty$ the two-sided parabolic ball $P(0,R)$ centered at the origin does not contain points modified by surgeries for infinitely many $j$, then the conclusion (a) holds. After passing to a subsequence we can thus assume there is some $R<\infty$, such that $P(0,R)$ contains points modified by by surgeries for all $j$. Since $\Lambda=\infty$, arguing similarly as in the proof of Claim \ref{curv_bounds} (curvature bounds) we see that for each $\rho<\infty$ we have
\begin{equation}
\limsup_{j\to\infty} \sup_{ \partial \widetilde{\mathcal{K}}^j\cap P(0,\rho)}|\nabla^\ell A|=0.
\end{equation}
It follows that conclusion (b) holds. This finishes the proof of the proposition.
\end{proof}

A large portion of this paper will deal with analyzing the limits constructed above:

\begin{definition}[generalized limit flow]\label{def_genlim}
A \emph{$\Lambda$-generalized limit flow} is any limit $(\K,\M)$ provided by Theorem \ref{prop_compactness} (hybrid compactness) or Proposition \ref{prop_LI} (large blowups).
\end{definition}

To conclude this section, let us record some properties of generalized limit flows:

\begin{corollary}[blowups of generalized limit flows]\label{cor_blowup}
Any blowup sequence to a generalized limit flow satisfies the conclusion of Proposition \ref{prop_LI} (large blowups).
\end{corollary}

\begin{proof}
This follows immediately from the fact that a limit of a limit is a limit.
\end{proof}

\begin{corollary}[one-sided minimization for generalized limit flows]\label{prop_onesidedmin}
Let $(\mathcal{K},\mathcal{M})$ be a generalized limit flow (see Definition \ref{def_genlim}). Suppose $\gamma$ is a $1$-cycle that at some time $t$ bounds a $2$-chain in $\partial K_t$, and $\mathcal{H}^2(\mathrm{sing}\, \partial K_t)=0$.
Then, there exists a $2$-chain $\Sigma$ such that:
\begin{enumerate}[(a)]
 \item $\partial \Sigma = \gamma$.
 \item $\mathcal{H}^2(\Sigma) \leq \mathcal{H}^2(\Sigma')$ for any $2$-chain $\Sigma'$ bounded by $\gamma$.
 \item $\Sigma$ is supported in $K_t$.
\end{enumerate}
\end{corollary}

\begin{proof} This follows from the one-sided minimization for smooth flows (see Proposition \ref{basic_prop}), via the same argument as in the the proof of \cite[Theorem 6.1]{White_size}.
\end{proof}

\bigskip

\section{Excluding microscopic surgeries}

Suppose $\widetilde{\K}^j=(\widetilde{\K}^j,\widetilde{\M}^j)$ is a $\Lambda$-blowup sequence that Hausdorff converges to a (quasi)static multiplicity-one plane $\mathcal{K}$. If $\Lambda>0$, then clearly any parabolic ball for $j$ large enough is unmodified by surgeries. The goal of this section is to rule out surgeries in the case $\Lambda=0$:

\begin{theorem}[no microscopic surgeries]\label{thm_no_microscopic}
Suppose $\widetilde{\K}^j$ is a $0$-blowup sequence (with $\delta\leq\bar{\delta}$ small enough) that Hausdorff converges in space-time to a (quasi)static multiplicity-one plane $\mathcal{K}$. Then, for any $R<\infty$ there exists a $j_0=j_0(R)<\infty$, such that for all $j\geq j_0$ the two-sided parabolic ball $P(0,R)$ contains no points modified by surgeries.
\end{theorem}

\begin{proof}
If the assertion failed, then after a controlled space-time rigid motion we would obtain a $0$-blowup sequence $\widetilde{\K}^j$ that Hausdorff converges to $\{x_1\leq 0\}\times (-\infty,0]$, but such that there are surgeries at scale $r_j\to 0$ centered at $(0,0)$. We will now argue similarly as in the proof of the local curvature estimate in \cite{HK}.
Set $t_0^j=r_j^2/2$, and consider Huisken's quantity
\begin{equation}
 \Theta\big(\widetilde{\mathcal{M}}^j,(0,t_0^j),\tau\big)=\int_{\partial { K}^j_{t_0^j-\tau}} \theta(x,\tau)\, dA(x),\qquad \theta(x,t)=\frac{1}{4\pi\tau} \exp\left({-\frac{|{x}|^2}{4\tau}}\right).
\end{equation}
For small backwards time, say $\tau_j=r_j^2$, we are $\delta$-close to a neck, and thus get
\begin{equation}\label{lowerdensity}
\liminf_{j\to\infty} \Theta\big(\widetilde{\mathcal{M}}^j,(0,t_0^j),r_j^2\big)>3/2.
\end{equation}
By Huisken's monotonicity formula \cite{Huisken_monotonicity} the function $\tau\mapsto \Theta\big(\widetilde{\mathcal{M}}^j,(0,t_0^j),\tau\big)$ is monotone if there are no surgeries. Note also that discarding connected components has the good sign.

Suppressing the index $j$ in the notation, we say that a surgery center $x_i$ at time $t_i$ is in the \emph{nonoscillating regime},
if 
\begin{equation}\label{eq_regime}
\frac{r|{x_i}|}{|{t_0-t_i}|}< \eps
\end{equation}
The change of the Huisken density due to any surgery $(x_i,t_i)$ in the nonoscillating regime has the good sign, provided $\eps$ is sufficiently small, namely setting $B_i=B(x_i,5\Gamma r_i)$ we have
\begin{equation}
 \int_{\partial  K_{t_i}^{\sharp} \cap B_i} \theta dA\leq \int_{\partial  K_{t_i}^{-}\cap B_i} \theta dA.
\end{equation}
Indeed, this follows from the fact that area decreases by a definite factor under surgery and the observation that we can make the ratio
between $\sup_{x\in B_i}\theta(x,t_0-t_i)$ and $\inf_{x\in B_i}\theta(x,t_0-t_i)$ as close as we want to $1$ by choosing $\eps$ small enough.
To estimate the cumulative error in Huisken's monotonicity inequality due to the surgeries $(x_i,t_i)$ violating \eqref{eq_regime}, in light of the spatial separation of surgeries from Proposition \ref{basic_prop}, it suffices to estimate the sum
\begin{equation}\label{sumtoest}
 \sum_i\frac{1}{\tau_i}e^{-|{x_i}|^2/5\tau_i}A_i,
\end{equation}
where $\tau_i=t_0-t_i$ and $A_i$ is the area of the region modified by the surgery. Here, we used that $|x|^2\geq\tfrac{4}{5}|{x_i}|^2$ for $x$ in the region around $x_i$ modified by surgery.
To estimate (\ref{sumtoest}), we first pull out a factor $e^{-|{x_i}|^2/10\tau_i}$.
Using Proposition \ref{basic_prop} again, we observe that the minimum of $\frac{|{x_i}|^2}{10\tau_i}+\log \tau_i$ over $\tau_i$
under the constraint $\tau_i\leq \frac{r|{x_i}|}{\eps}$ is attained at  $\tau_i= \frac{r|{x_i}|}{\eps}$. Thus
\begin{equation}\label{eq_compsum}
 \sum_i\frac{1}{\tau_i} e^{-\frac{|{x_i}|^2}{5\tau_i}}A_i\leq\gamma
 \sum_i \frac{\eps}{r |{x_i}|}e^{-\frac{\eps |{x_i}|}{10r}}A_i,
\end{equation}
where $\gamma:=\sup_i e^{-|{x_i}|^2/10\tau_i}$ is small, since
$
\frac{|{x_i}|^2}{\tau_i}=\frac{r |{x_i}|}{\tau_i}\frac{ |{x_i}|}{r}\geq \tfrac{\eps}{10}\delta^{-1} 
$
is large, provided $\bar{\delta}$ is small enough, again by Proposition \ref{basic_prop}.
Finally, using again that the regions modified by surgeries are separated by a large multiple of $r$  and have area comparable to $r^2$, the sum on the right hand side of \eqref{eq_compsum} can be uniformly estimated by a multiple of 
\begin{equation}
\int_{r}^\infty \frac{1}{r R}e^{-R/r}R\, dR=\int_{1}^\infty e^{-u}du<e.
\end{equation}
Thus, the cumulative error due to surgeries violating \eqref{eq_regime} is less than $1/100$, provided $\bar{\delta}$ is small enough.
This proves almost monotonicity, and hence by closeness to a multiplicity-one plane at scale $\tau=1$ for $j$ large enough gives the desired contradiction with \eqref{lowerdensity}. 
\end{proof}

\bigskip

\section{Multiplicity-one}

In this section, we prove that every generalized limit flow  has multiplicity one. To this end, we will adapt the arguments from White \cite{White_size} to our setting of flows with surgeries.

\subsection{Large blowups with entropy at most two}

Fixing the initial domain $K_0$, we consider the class $\mathcal{C}$ of all limits $(\mathcal{K},\mathcal{M})$ given by case (a) of Proposition \ref{prop_LI} (large blowups) that are not (quasi-)static multiplicity-two planes, and have entropy at most 2, namely
\begin{equation}
\lim_{\tau\to\infty}\Theta(\mathcal{M},0,\tau)\leq 2.
\end{equation}

\begin{proposition}[partial regularity]
\label{prop:partial-reg}
For $(\mathcal{K},\mathcal{M})\in\mathcal{C}$ no tangent flow at a singular point can be static or quasi-static. In particular, the singular set has dimension at most $1$.\footnote{We use the convention that we call a point \emph{regular} if the flow is smooth with multiplicity-one in a backwards parabolic neighborhood (in particular, this allows for quasi-static multiplicity-one planes).}
\end{proposition}
\begin{proof}
By the equality case of the monotonicity formula and the definition of the class $\mathcal{C}$ no tangent flow can be a (quasi-)static plane of higher multiplicity. Together with standard stratification \cite{White_stratification}, remembering that we are working in $\mathbb{R}^3$, the assertion follows.
\end{proof}

\begin{corollary}[static and quasi-static case]
\label{lem:staticC}
If $(\mathcal{K},\mathcal{M})\in \mathcal{C}$ is (quasi-)static, then either (i) $\mathcal{K}$ is a (quasi-)static halfspace and $\mathrm{spt}\mathcal{M}$ is the (quasi-)static plane $\partial\mathcal{K}$, or (ii) $\mathrm{spt}\mathcal{M}$ is a pair of (quasi-)static parallel planes and $\mathcal{K}$ is the region in between.
\end{corollary}
\begin{proof}
Since it is (quasi-)static, $\partial\mathcal{K}$ must be smooth and flat by Proposition \ref{prop:partial-reg} (partial regularity), and hence a union of one or two disjoint (quasi-)static multiplicity-one planes. The result then follows from one-sided minimization (Corollary \ref{prop_onesidedmin}).
\end{proof}

\begin{theorem}[{separation theorem, c.f. \cite[Theorem 7.4]{White_size}}]
\label{thm:sep}
Let $(\mathcal{K},\mathcal{M})\in\mathcal{C}$. If there is some plane $V$, such that $\bigcap_t K_t\supseteq V$ and  the complement of $\bigcap_t K_t$ contains points on each side of $V$, then $(\mathcal{K},\mathcal{M})$ is static and $\mathcal{K}$ is the region between two parallel planes.
\end{theorem}

\begin{proof}
Using Proposition \ref{prop:partial-reg} (partial regularity) and  Corollary \ref{prop_onesidedmin} (one-sided minimization for generalized limit flows) we can follow the proof of \cite[Theorem 7.4]{White_size} to show that the Brakke flow $\mathcal{M}$ splits into two components $\mcfM_\pm$, each contained in the respective halfspace defined by $V$. 
Since each Brakke flow $\mcfM_\pm$ has density at least 1, but the sum of densities is at most 2, each $\mcfM_\pm$ must have density exactly 1 and hence be a multiplicity-one plane.
Finally, since $K_t \supseteq V$ is in particular nonempty for all $t$, we conclude that each plane $\mcfM_\pm$ is static.
\end{proof}

Now, as in \cite[Section 4]{White_size}, the relative thickness of a set $S$ in $B(x,r)$ is defined by
\begin{equation}
\thi(S,x,r)= \frac{1}{r} \inf_{|v|=1} \sup_{y\in S\cap B(x,r)} |\langle v,y-x\rangle|.
\end{equation} 

\begin{theorem}[{Bernstein-type theorem, c.f. \cite[Theorem 7.5]{White_size}}]
\label{thm:bernstein}
There exists an $\varepsilon>0$ with the following significance. If $(\mathcal{K}, \mathcal{M}) \in \mathcal{C}$ and there is a point $x$ such that
\begin{equation}\label{ass_thin}
\limsup_{r\to \infty} \thi(K_{-r^2},x,r)<\varepsilon,
\end{equation}
and
\begin{equation}\label{ass_nonexp}
\liminf_{r\to\infty} \frac{\dist(K_{r^2},x)}{r} <1,
\end{equation}
then $\mathcal{M}$ is a pair of static parallel multiplicity-one planes and $\mathcal{K}$ is the region in between.
\end{theorem}

\begin{proof}
Since $\mathcal{K}$ is nested and satisfies the avoidance principle, choosing $\eps$ small enough, we can apply the expanding hole result from \cite[Corollary 4.3]{White_size}, to infer, taking also into account the assumption \eqref{ass_nonexp}, that $S:=\bigcap_t K_t$ is nonempty. Now, consider the flows obtained by translating $(\mathcal{K},\mathcal{M})$ by $(0,-T)$ and let $(\mathcal{K}',\mathcal{M}')$ be a limit as $T\to\infty$. Then $(\mathcal{K}',\mathcal{M}')$ is a static flow with $K_t'=S$ (at any time $t$). Note that either  $(\mathcal{K}',\mathcal{M}')$ is a multiplicity-two plane, or belongs to the class $\mathcal{C}$ and hence by Corollary \ref{lem:staticC} (static or quasi-static case) and assumption \eqref{ass_thin} must be a slab. In both cases we can apply Theorem \ref{thm:sep} (separation theorem) to conclude the proof. 
\end{proof}

\subsection{Sheeting theorem for $\Lambda$-blowup sequences}

In this subsection, we prove a sheeting theorem for blowup sequences resulting in non-microscopic surgeries. We start with the following lemma, which will be used to find a separating hypersurface:

\begin{lemma}[slab rescaling]
\label{lem:sheet0}
Let $(\widetilde{\mathcal{K}}^j, \widetilde{\mathcal{M}}^j)$ be a $\Lambda$-blowup sequence with $\Lambda>0$, and suppose that
\begin{equation}\label{eq_ass_pl}
d_{\mathrm{H}}\big(\widetilde{\mathcal{K}}^j \cap P(0,2), (V\times\mathbb{R}) \cap P(0,2)\big)\to 0,
\end{equation}
for some plane $V$ through the origin. Then, there exist $\mu_j\to \infty$, such that $\mathcal{D}_{\mu_j} \widetilde{\mathcal{M}}^j$ converges smoothly to a pair of parallel planes, and $\mathcal{D}_{\mu_j} \widetilde{\mathcal{K}}^j$ converges to the region in between. 
\end{lemma}
\begin{proof}
Fixing $\eps>0$ small enough, let $\mu_j<\infty$ (c.f. the expanding hole lemma \cite[Corollary 4.2]{White_size}) be the largest number such that for all $r\in [\mu_j^{-1},1]$ we have
\begin{equation}\label{eq_how_thick}
\textrm{th}(\tilde{K}^j_{-r^2},0,r)\leq \varepsilon\qquad \textrm{and} \qquad d(\tilde{K}^j_{r^2},0)\leq r.
\end{equation}
Assumption \eqref{eq_ass_pl} implies $\mu_j\to \infty$, so remembering $\Lambda>0$ we have in particular
\begin{equation}
\frac{\mu_j \lambda_j}{H^j_{\textrm{neck}}} \to \infty.
\end{equation} 
By Proposition \ref{prop_LI} (large blowups) we can take a subsequential limit  of $\mathcal{D}_{\mu_j}(\widetilde{\mathcal{K}}^j, \widetilde{\mathcal{M}}^j)$. Any such limit $(\mathcal{K},\mathcal{M})$ satisfies
\begin{equation}\label{eq_how_thick_lim}
\textrm{th}(K_{-r^2},0,r)\leq \varepsilon\qquad \textrm{and} \qquad d({K}_{r^2},0)\leq r \qquad \textrm{for all } r\geq 1,
\end{equation}
with at least one inequality being non-strict for $r=1$. In particular, we must be in case (a) of  Proposition \ref{prop_LI}, and hence our limit is a weakly mean-convex Brakke flow without surgeries.
Moreover, by Corollary \ref{prop_onesidedmin} (one-sided minimization) and monotonicity the density at infinity of $\mathcal{M}$ is at most 2. Hence,  Theorem \ref{thm:bernstein}  (Bernstein-type theorem) implies that $\mathcal{M}$ consists of separate multiplicity-one planes and $\mathcal{K}$ is the region in between. Finally, the local regularity theorem \cite{white_regularity} gives smooth convergence as desired. This proves the lemma.
\end{proof}

Now if $(\widetilde{\mathcal{K}}^j, \widetilde{\mathcal{M}}^j)$ is a $\Lambda$-blowup sequence with $\Lambda>0$ that converges to a multiplicity-two plane, we construct a separating surface as follows.
Let $S^j_t$ be the set of centres of open balls $B$ such that $B \subset K^j_t$ and $\bar{B}$ touches $\partial K^j_t$ at two or more points. Set $\mathcal{S}^j = \bigcup_t S^j_t$. 

\begin{theorem}[sheeting theorem]
Let $(\widetilde{\mathcal{K}}^j, \widetilde{\mathcal{M}}^j)$ be a $\Lambda$-blowup sequence with $\Lambda>0$, and suppose that
\begin{equation}\label{eq_ass_pl2}
d_{\mathrm{H}}\big(\widetilde{\mathcal{K}}^j \cap P(0,4), (V\times\mathbb{R}) \cap P(0,4)\big)\to 0,
\end{equation}
for some plane $V$ through the origin. Then, for $j$ sufficiently large, $\mathcal{S}^j\cap P(0,1)$ is a smooth hypersurface that divides $\partial \widetilde{\mathcal{K}}^j$ into two nonempty components. In particular, any convergent subsequence of $
\widetilde{\mathcal{M}}^j$  converges smoothly in $P(0,1)$ to a multiplicity-two plane. 
\end{theorem}
\begin{proof}
Observe first that thanks to \eqref{eq_ass_pl2} and $\Lambda>0$, for $j$ large enough there are no points modified by surgery in $P(0,3)$. 

Now, suppose towards a contradiction that there are points $X_j \in \mathcal{S}^j\cap P(0,1)$ about which $\mathcal{S}^j$ fails to be a smooth properly embedded surface. Consider the translates $({\widehat{\mathcal{K}}}^j,\widehat{\mathcal{M}}^j) = (\widetilde{\mathcal{K}}^j-X_j, \widetilde{\mathcal{M}}^j-X_j)$. Up to taking a subsequence, $\widehat{\mathcal{K}}^j \cap P(0,2)$ will Hausdorff converge to a translate of $V\times \mathbb{R}$.  By Lemma \ref{lem:sheet0} (slab rescaling) for large enough $j$ there exist radii $r_j>0$ such that $\widehat{\mathcal{K}}^j\cap P(X_j, r_j)$ is given by the region in $P(X_j,r_j)$ bounded by two smooth, disjoint hypersurfaces $\mathcal{S}^j_1, \mathcal{S}^j_2$, which each converge smoothly to $V\times \mathbb{R}$. But $\mathcal{S}^j \cap P(X_j,r_j/2)$ is clearly given by the equidistant set between $\mathcal{S}^j_1$ and $\mathcal{S}^j_2$, which is a smooth hypersurface for large enough $j$, contradicting the choice of $X_j$.

Thus, $\mathcal{S}^j\cap P(0,1)$ is a smooth properly embedded surface, and hence divides $P(0,1)$ into two disjoint open subsets. Consider the components $\mathcal{M}^j_1, \mathcal{M}^j_2$ of $\partial\mathcal{K}^j$ in each respective subset of $P(0,1)$. Each component Hausdorff converges in $P(0,1)$ to $V\times \mathbb{R}$ with some integer multiplicity. But by Corollary \ref{prop_onesidedmin} (one-sided minimization) the sum of densities is at most 2. Hence, we conclude that both components converge smoothly with multiplicity-one.
\end{proof}

For later applications, it will be convenient to restate the smooth convergence provided by the sheeting theorem in the following graphical form:

\begin{corollary}[sheeting theorem; restated]
\label{thm:sheet2}
For every $\eta>0$ there exists $\varepsilon=\varepsilon(\eta,K_0)>0$ with the following significance.
Let $(\mcfK', \mcfM')$ be either (i) a dilation $\mathcal{D}_\lambda(\K-X, \M-X)$ of a $(\delta,\mathcal{H})$-flow $(\K,\M)$ starting at $K_0$ with $\delta\leq\bar{\delta}$ and
$\min\big(H_{\textrm{thick}} , \frac{  H_{\textrm{neck}}}{H_{\textrm{thick}} }, \frac{  H_{\textrm{trigger}}}{H_{\textrm{neck}} }   \big)> \eps^{-1}$ such that
 $\lambda H_{\textrm{neck}}^{-1} >\eta$, or (ii) a $\Lambda$-generalized limit flow with $\Lambda>\eta$. If furthermore
\begin{equation}
d_{\mathrm{H}}(\mathcal{K}' \cap P(0,4),( \{x_3=0\}\times\mathbb{R} )\cap P(0,4)) < \varepsilon,
\end{equation}
then there exist functions $f , g  : P(0,3) \to \mathbb{R}$ satisfying:
\begin{enumerate}
\item $f\leq g$;
\item $f,g$ have $C^{2,1}$ norm at most $\eta$;
\item Inside $P(0,2)$, the set $\mathcal{K}'$ is given by the region between $\graph(f)$ and $\graph(g)$;
\item $f,g$ satisfy the graphical mean curvature flow equation;
\item For each fixed $x$, the functions $t\mapsto f(x,t)$ and $t\mapsto g(x,t)$ are increasing and decreasing, respectively.
\end{enumerate}
\end{corollary}

\subsection{Ruling out generalized limit flows with multiplicity-two}
As above, fixing the initial condition $K_0$, we consider all generalized limit flows $(\mathcal{K},\mathcal{M})$ as in Definition \ref{def_genlim}. We recall that they arise as limits of $\Lambda$-blowup sequences, where $\Lambda\in(0,\infty]$.

Similarly as in \cite[Section 9]{White_size}, if $\mathcal{K}'$ is any closed subset of spacetime we denote by $\phi(\mathcal{K}')$ the infimum of $s>0$ such that
\begin{equation}
d_{\mathrm{H}}\left(\mathcal{K}'\cap P(0,1/s),(V\times\mathbb{R})\cap P(0,1/s)\right) < s
\end{equation}
for some plane $V$ through the origin.

\begin{lemma}[isolation]\label{lemma_isolation}
There exists an $\eps>0$, such that if $\K'$ is a tangent flow to a generalized limit flow and $\phi(\mcfK')<\eps$, then $\phi(\mcfK')=0$.
\end{lemma}
\begin{proof}
Consider a sequence of tangent flows $(\K^j, \M^j)$ to $\Lambda_j$-generalised limit flows, $\Lambda_j>0$, such that $\phi(\mcfK^j)\to 0$. It is enough to show that $\phi(\mcfK^j)=0$ for all sufficiently large $j$.

Note that in particular each $(\K^j, \M^j)$ is an $\infty$-generalised limit flow. Hence, by the sheeting theorem (Corollary \ref{thm:sheet2}) for large enough $j$ associated to our sequence there are functions $f_j, g_j $, defined on an exhaustion of $\mathbb{R}^n\times (-\infty,0)$, such that:
\begin{enumerate}
\item either $f_j< g_j$ everywhere, or $f_j\equiv g_j$;
\item both sequences $f_j$ and $g_j$ converge smoothly on compact subsets to 0;
\item For any $U\subset\subset \mathbb{R}^{n+1}$ and $[a,b]\subset (-\infty,0)$, for $j$ large enough the region $K^j_t$ coincides in $U$, after a suitable rotation, with the region between $\graph(f_j)$ and $\graph(g_j)$, for all $t\in [a,b]$;
\item $f_j,g_j$ are solutions of the graphical mean curvature flow equation;
\item $f_j, g_j$ are increasing and decreasing in time, respectively.
\end{enumerate}
Moreover, since $\Lambda_j>0$, by monotonicity each tangent flow is backwardly self-similar, so 
\begin{equation}\label{eq:self-sim-fn} f_j(rx,r^2t) = rf_j(x,t), \qquad g_j(rx,r^2t)=rg_j(x,t)\end{equation} for all $r>0$ and all $(x,t)$ with $t<0$. 

Now, if $f_j<g_j$ for infinitely many $j$, then using the Harnack inequality similarly as in \cite[Case 1 in the proof of Theorem 9.1]{White_size} we can find $c_j>0$ such that $c_j(g_j-f_j)$ subsequentially converges on compact subsets to the constant function $u\equiv 1$ on $\mathbb{R}^n \times (-\infty,0)$. However, using (\ref{eq:self-sim-fn}) we infer that $u(rx,r^2t) = ru(x,t)$, which is a contradiction.

Thus, $f_j\equiv g_j$ for all large enough $j$. But then the functions are constant in $t$, and together with the self-similarity we infer that $f_j\equiv g_j$ is 1-homogenous. Remembering smoothness this implies linearity. We conclude for all sufficiently large $j$, that $\K^j$ is a plane and in particular $\phi(\K^j)=0$. 
\end{proof}

\begin{lemma}[minimal surface]\label{lemma_min_surf}
Let $(\mathcal{K},\mathcal{M})$ be a generalized limit flow. Suppose one of the tangent flows of $(\mathcal{K},\mathcal{M})$ at some space-time point $X_0=(x_0,t_0)$ is a static or quasistatic multiplicity-two plane. Then there exists a spatial neighborhood $W$ of $x_0$, a properly embedded smooth minimal surface $\Sigma$ in $W$, and an interval $(a,b)$ with $a<t_0$ such that
\begin{equation}
\partial K_t\cap W=\Sigma\quad \textrm{for all } t\in (a,b).
\end{equation}
Moreover, we have the Gaussian density lower bound $\Theta(\mathcal{M},\cdot)\geq 2$ on all of $\Sigma\times (-\infty,b]$.\footnote{Recall that the Gaussian density is defined as $\Theta(\mathcal{M},X):=\lim_{\tau\to 0}\Theta(\mathcal{M},X,\tau)$.}
\end{lemma}

\begin{proof}Note that there are no surgeries in a space-time neighborhood of $X_0$. Hence, the existence of $\Sigma$ follows from the sheeting theorem (Corollary \ref{thm:sheet2}) exactly as in White \cite[proof of Theorem 12.1]{White_size}, using the argument of \cite[proof of Theorem 9.2]{White_size} to rule out the case $f_j<g_j$. 

For the density, let $\mathcal{Z}$ be the set of points of $\mathcal{M}$ at which none of the tangent flows is planar (which is necessarily in the complement of the surgery region). 
By general stratification results \cite{White_stratification},  the parabolic Hausdorff dimension of $\mathcal{Z}$ is at most $1$. In particular, the spatial projection $\pi(\mathcal{Z})$ has Hausdorff dimension at most $1$. So by upper semicontinuity of the density, it is enough to show that $\Theta(\mathcal{M}, (y,t)) =2$ for all $y\in \Sigma\setminus \pi(\mathcal{Z})$ and $t<b$. 

Fix $y\in \Sigma\setminus \pi(\mathcal{Z})$, and let $T^*=\sup\{\tau<b\, |\, \Theta(\mathcal{M},(y,\tau))\neq 2\}$. Clearly $T^*\leq a$. Suppose towards a contradiction that $T^*>-\infty$. Note that a neighborhood of $(y,T^*)$ is unmodified by surgeries, since otherwise the density near $(y,T^\ast)$ would be less than $2$. Now, consider a tangent flow $(\mathcal{K}' , \mathcal{M}')$ at $(y,T^*)$. If it is a static multiplicity 2 plane, then applying the first part of the theorem shows that the density is 2 for times close to $T^*$, which contradicts the definition of $T^*$. 
If $(\mathcal{K}' , \mathcal{M}')$ was a quasistatic plane or a static plane of multiplicity 1, then we would obtain a contradiction with the fact that $\Theta(\mathcal{M}, (y,\tau))=2$ for $\tau \in [T^*, b)$. Thus, $(y,T^*)\in\mathcal{Z}$, contradicting the choice of $y$. This proves the lemma.
\end{proof}

\begin{theorem}[multiplicity-one for generalized limit flows]\label{thm_mult_one}
Given any mean-convex initial data, static or quasistatic multiplicity-two planes cannot occur as generalized limit flows.
\end{theorem}

\begin{proof}
Given a mean-convex initial condition $K_0$, consider any $\Lambda$-blowup sequence $(\widetilde{\K}^j,\widetilde{\M}^j)$. By Proposition \ref{prop_L0} (microscopic surgeries) and White's multiplicity-one theorem for blowup limits of mean-convex level set flow \cite{White_size} we may assume that $\Lambda>0$. Now, suppose towards a contradiction that $(\widetilde{\K}^j,\widetilde{\M}^j)$ converges to a static or quasistatic multiplicity-two plane. We may assume that the space-time origin is contained in $\partial\widetilde{\K}^j$ for every $j$.\\

Recall that by definition of generalized limit flows we have $\widetilde{\K}^j=\mathcal{D}_{\lambda_j}(\mcfK^j-X_j)$, where $\mcfK^j$ is a sequence of $(\delta,\mathcal{H}^j)$-flows with initial condition $K_0$, and where $\mathcal{H}^j\to\infty$ and $\lambda_j\to \infty$. Denoting by $\eps>0$ the smaller one of the two constants from Lemma \ref{lemma_isolation} (isolation) and \cite[Theorem 9.1]{White_size}, let $\mu_j>0$ be the largest number such that
\begin{equation}
\phi(\dilD_{\mu_j}\widetilde{\K}^j) \geq \eps/2.
\end{equation}
Note that $\mu_j\to 0$. After passing to a subsequence, we can also assume that $\dilD_{\mu_i}\widetilde{\K}^j$ Hausdorff converges to a limit $\mathcal{K}$. By construction, we have
\begin{equation}\label{resc_ass}
\phi(\mcfK)\geq \eps/2\qquad \mathrm{ but } \qquad \phi(\dilD_{\lambda}\mcfK) \leq \eps/2 \,\, \mathrm{for} \,\, \lambda\geq 1.
\end{equation}

\bigskip

If $ \mu_j\lambda_j /H^{j}_{\textrm{neck}}\to 0$, then by Proposition \ref{prop_L0} (microscopic surgeries) we would see that $\mcfK$ is a blowup limit or homothetic copy of the level set flow. But then, using \eqref{resc_ass} and arguing similarly as in \cite[proof of Theorem 12.3]{White_size}, we could construct a blowup limit of the level set flow that is a non-planar minimal cone, which gives the desired contradiction.\\

Thus, after passing to a subsequence we can assume that $\mu_j\lambda_j/H^{j}_{\textrm{neck}}\to \Lambda'>0$. In particular, $\mu_j\lambda_j\to\infty$ and $\mcfK$ is a generalized limit flow that comes with a family of Radon measures $\mathcal{M}$, which is provided by Theorem \ref{prop_compactness} (hybrid compactness) or Proposition \ref{prop_LI} (large blowups), respectively. Now, by Lemma \ref{lemma_isolation} (isolation) and \eqref{resc_ass}, any tangent flow to $(\K,\M)$ at the space-time origin must be a static or quasistatic multiplicity-two plane. So we can apply Lemma \ref{lemma_min_surf} (minimal surface) to obtain a minimal surface $\Sigma$ containing the origin and a real number $b$, such that
\begin{equation}
\Theta(\M,\cdot)\geq 2\,\,\mathrm{on}\,\, \Sigma\times (-\infty,b].
\end{equation}
In particular, it follows that $\Lambda'<\infty$. Now translate $(\K,\M)$ by $(x,t)\mapsto (x,t+j)$ and using Theorem \ref{prop_compactness} (hybrid compactness) along  $j\to\infty$ pass to a subsequential limit to get a static Brakke $\delta$-flow $(\K',\M')$ satisfying
\begin{equation}
K'_t = \bigcup_{s}K_s\quad \forall t,\qquad\textrm{and}\qquad
\Theta(\M',\cdot)\geq 2\,\,\mathrm{on}\,\, \Sigma\times (-\infty,\infty).
\end{equation}
Observe that since the flow is static, it does not contain any surgeries. Moreover, by \eqref{resc_ass} it is not planar. Hence, taking a tangent flow at $-\infty$ to $(\K',\M')$ we obtain a non-planar static minimal cone, which gives the desired contradiction. This proves the theorem.
\end{proof}

\bigskip

\section{Partial regularity and convexity}

In this section, we adapt some of the arguments from \cite{White_nature} to our setting to show that generalized limit flows have small singular set and nonnegative second fundamental form.

\begin{theorem}[partial regularity]\label{thm_part_reg}
The singular set of any generalized limit flow has parabolic Hausdorff dimension at most $1$.
\end{theorem}

\begin{proof}
Given any generalized limit flow $(\mathcal{K},\mathcal{M})$, it follows from Corollary \ref{prop_onesidedmin} (one-sided minimization for generalized limit flows) and Theorem \ref{thm_mult_one} (multiplicity-one for generalized limit flows) that nontrivial cones and higher-multiplicity planes cannot occur as tangent flows of $(\mathcal{K},\mathcal{M})$. Hence, by standard dimension reduction and the local regularity theorem \cite{White_stratification,white_regularity}, the singular set has parabolic Hausdorff dimension at most $1$.
\end{proof}

\begin{proposition}[rigidity]\label{prop_rig}
Let $(\mathcal{K},\mathcal{M})$ be a generalized limit flow. If $0\in\partial K_{0}$ is a regular point, and $H(0,0)=0$, then $(\mathcal{K},\mathcal{M})\cap \{t\leq 0\}$ is a flat multiplicity-one plane.
\end{proposition}

\begin{proof}
Suppose that $(0,0)$ is a regular point and $H(0,0)=0$. Then by the strict maximum principle we have $H\equiv 0$ in some backwards parabolic ball $P^-(0,\rho)$. By Theorem \ref{thm_part_reg} (partial regularity) we can choose a time $t_0\in (-\rho^2,0)$ at which the solution is completely smooth. Then, again by the strict maximum principle, there is an entire connected component $\Sigma\subset M_{t_0}$ that contains the origin and on which the mean curvature vanishes identically. Note that $\Sigma$ must be noncompact, since there are no compact minimal surfaces in Euclidean space. Next, again by the smallness of the singular set any $X\in \Sigma\times (-\infty,t_0]\setminus\textrm{sing}\mathcal{M}$ can be connected to $(0,t_0)$ by a time-like space-time curve that entirely avoids the singular set. Together with the strict maximum principle this yields
\begin{equation}
\Sigma\times (-\infty,t_0]\subseteq \partial\mathcal{K}.
\end{equation}
In particular, there are no surgeries near $\Sigma$.
Now, by Corollary \ref{prop_onesidedmin} (one-sided minimization for generalized limit flows) and Theorem \ref{thm_mult_one} (multiplicity-one for generalized limit flows) in the case $\Lambda>0$, respectively Proposition \ref{prop_L0} (convergence to level-set flow) and the one-sided minimization and multiplicity-one theorem for blowups of the level-set flow from \cite{White_size,White_nature} in the case $\Lambda=0$, the tangent cone at infinity of $\Sigma$ must be a multiplicity-one plane. Hence, by monotonicity, $\Sigma$ itself is a multiplicity-one plane. Finally, by White's strong halfspace result \cite[Theorem 7]{White_nature}, which is stated for set-theoretic subsolutions and hence applicable in our setting, there cannot be any other connected components, and the assertion follows.
\end{proof}

\begin{theorem}[nonnegative second fundamental form]\label{thm_conv}
Let $(\mathcal{K},\mathcal{M})$ be a generalized limit flow. If $(0,0)$ is a regular point, then all principal curvatures at the origin are nonnegative.
\end{theorem}

\begin{proof}
Fixing $K_0\subset\mathbb{R}^3$, suppose towards a contradiction that there is a sequence of generalized limit flows $(\mathcal{K}^j,\mathcal{M}^j)$ and a sequence of regular points $X_j$ such that $\frac{\lambda_1}{H}(X_j)$ converges to an infimal value $\gamma< 0$. Note that $\gamma>-\infty$ thanks to the bound $|A|\leq CH$ from Proposition \ref{basic_prop}. Moreover, by translating and scaling we may assume that $X_j=0$ and
\begin{equation}\label{rescaling_white}
\sup_{P(0,1)} |A_{\partial\widetilde{\mathcal{K}}^j}| \leq 1\leq \sup_{\overline{P(0,1)}} |A_{\partial\widetilde{\mathcal{K}}^j}|.
\end{equation}
If there is no $r>0$ such that the flow is unmodified by surgeries in $P(0,r)$, then after adjusting our sequence, using in particular item (c) of Definition \ref{def_replacestd} (replacing a $\delta$-neck by standard caps), we may assume that $(0,0)$ lies in the pre-surgery domain.

If $(\mathcal{K}^j,\mathcal{M}^j)$ is a $0$-blowup sequence, then using Proposition \ref{prop_L0} (convergence to level-set flow) and Theorem \ref{thm_no_microscopic} (no microscopic surgeries) we obtain contradiction with White's convexity theorem for blowup limits of the level-set flow \cite[Theorem 8]{White_nature}. Hence, by Theorem \ref{prop_compactness} (hybrid compactness) and Proposition \ref{prop_LI} (large blowups) we may assume that $(\mathcal{K}^j,\mathcal{M}^j)$ converges to a generalized limit flow $(\mathcal{K},\mathcal{M})$. By \eqref{rescaling_white} and Proposition \ref{prop_rig} (rigidity) the limit $(\mathcal{K},\mathcal{M})$ must have strictly positive mean curvature. Hence $\lambda_1/H$ attains a strictly negative minimum at the space-time origin. This contradicts the strict maximum principle, and thus proves the theorem.
\end{proof}

\bigskip

\section{Canonical neighborhoods and existence theorem}

In this final section, we prove the canonical neighborhood theorem and the existence theorem for flows with surgery. As before, we use the abbreviation
\begin{equation*}
\mathcal{H}^j\to \infty \quad:\Leftrightarrow\quad
\min\left(H^j_{\textrm{thick}} , \frac{  H^j_{\textrm{neck}}}{H^j_{\textrm{thick}} }, \frac{  H^j_{\textrm{trigger}}}{H^j_{\textrm{neck}} }   \right)\to \infty.
\end{equation*}

\begin{theorem}[canonical neighborhoods]\label{thm_can_nbd}
Suppose $\K^j$ is a sequence of $(\delta,\mathcal{H}^j)$-flows starting at a  smooth compact mean-convex domain $K_0\subset\mathbb{R}^3$, such that $\delta\leq\bar{\delta}$ and $\mathcal{H}^j\to \infty$. Then, for any sequence of space-time points $X_j\in\partial\K^j$ with $H(X_j)\to\infty$, the parabolically rescaled flows
$\mathcal{D}_{H(X_j)}( \K^j-X_j)
$ subsequentially converge to either
\begin{itemize}
\item the evolution of a standard cap preceded by a round shrinking cylinder, or 
\item a round shrinking cylinder, round shrinking sphere, translating bowl or ancient oval.
\end{itemize}
\end{theorem}

\begin{proof}
Consider the blowup sequence $\widetilde{\K}^j:=\mathcal{D}_{\lambda_j}( \K^j-X_j)$, where $\lambda_j$ is chosen such that
\begin{equation}\label{rescaling_white_again}
\sup_{P(0,1)} |A_{\partial\widetilde{\mathcal{K}}^j}| \leq 1\leq \sup_{\overline{P(0,1)}} |A_{\partial\widetilde{\mathcal{K}}^j}|.
\end{equation}
If $\lambda_j/H^j_{\textrm{neck}}\to 0$, then using Proposition \ref{prop_L0} (convergence to level-set flow) and White's structure theorem for blowup limits of the level set flow \cite[Theorem 1]{White_nature}  we see that $\widetilde{\K}^j$ subsequentially converges to a nontrivial ancient noncollapsed flow that is smooth until it becomes extinct. By the recent classification of Brendle-Choi \cite{BC} and Angenent-Daskalopoulos-Sesum \cite{ADS} any such limit $\mathcal{K}$ must be a round shrinking cylinder, round shrinking sphere, translating bowl or ancient oval. Similarly, if $\lambda_j/H^j_{\textrm{neck}}\to \infty$, then taking also into account Proposition \ref{prop_LI} (large blowups), we see that $\widetilde{\K}^j$ again subsequently converges to one of these four limits.\\

We can thus assume from now on that $(\widetilde{\K}^j,\widetilde{\mathcal{M}}^j)$, where $\widetilde{\mathcal{M}}^j$ is the associated family of Radon measures defined as in \eqref{ass_radon}, is a $\Lambda$-blowup sequence with $0<\Lambda<\infty$.   By Theorem \ref{prop_compactness} (hybrid compactness) we can pass to a subsequential limit $(\K,\M)$.
Define $\mathcal{K}^0=\{K_t^0\}_{t\leq 0}$ by for each $t\leq 0$ setting $K_t^0$ to be the connected component of $K_t$ that contains the origin. Note that all these sets in fact contain a closed ball $B$ of positive radius thanks to \eqref{rescaling_white_again}. By Theorem \ref{thm_part_reg} (partial regularity) together with Theorem \ref{thm_conv} (nonnegative second fundamental form) and connectedness, the sets $K_t^0$ are smooth and convex for almost every $t$. Remembering in particular the way surgeries are performed, we see that convexity in fact holds at all $t\leq 0$.
Now given any $p\in \partial K_t^0$ the convex hull of $p$ and $B$ is contained in $K_t^0$, and consequently, remembering Corollary \ref{prop_onesidedmin} (one-sided minimization for generalized limit flows) and Theorem \ref{thm_mult_one} (multiplicity-one for generalized limit flows), any tangent flow at $(p,t)$ must be a multiplicity-one plane. Hence, $K_t^0$ is smooth and convex for all $t\leq 0$.\\

If $\mathcal{K}^0$ does not contain surgeries, then by the classification from \cite{BC,ADS} it must be a round shrinking cylinder, round shrinking sphere, translating bowl or ancient oval. If $\mathcal{K}^0$ does contain a surgery, let $T\leq 0$ be a surgery time and let $N\subset {K}_T^0$ be a surgery neck of quality $\bar{\delta}$ sitting in the backward time slice.
Note that ${N}$ is the limit of some solid $\bar{\delta}$-necks $\widetilde{N}^j$ in the approximators $\widetilde{\K}^j$.
By part (b) of Definition \ref{def_MCF_surgery} (flow with surgery) we can find a curve $\gamma_j$ in the approximator connecting
$\{H=H_{\textrm{trig}}^j\}$ and $\{H\leq H_{\textrm{th}}^j\}$,  such that it passes through $\widetilde{N}^j$ but avoids all other $\bar{\delta}$-necks of the disjoint collection. We can assume that the curve $\gamma_j$ enters and leaves $\widetilde{N}^j$ exactly once. 
Note furthermore that $\gamma_j$ must intersect each of the two boundary disks of the cylinder exactly once.
Let $x_j$ be the center of $\widetilde{N}^j$. Since $\mathcal{K}^0$ is smooth with strictly positive mean curvature, and since ${H^j_{\textrm{trig}}}/{H^j_{\textrm{neck}}},{H^j_{\textrm{neck}}}/{H^j_{\textrm{th}}}\rightarrow \infty$, given any $R<\infty$, for $j$ large enough the curve $\gamma_j$ must start and end outside $B(x_j, R )$. Thus, ${K}_T^{0,-}\setminus N$ has at least two unbounded components. Since ${K}_T^{0,-}$ is connected, ${K}_T^{0,-}\setminus N$ must have exactly two components.
We have thus shown that ${K}_T^{0,-}$ has two ends, and consequently it contains a line, and all prior time slices contain this line as well.
Hence, at each fixed time the convex set splits off an $\R$-factor, and thus there cannot be any other surgeries. It follows that ${\K}^0$ is a round cylindrical flow for $t<T$. Similarly, by the uniqueness of the standard solution, ${\K}^0$ must be the evolution of the standard cap for $t>T$.\\

In particular, since all the ancient solutions from the classification sweep out the entire space for $t\to -\infty$ it follows that there are in fact no other connected components, i.e.
\begin{equation}
\K^0=\K.
\end{equation}
Finally, since $H(X_j)\to H(0)\in (0,\infty)$, remembering also Theorem \ref{thm_no_microscopic} (no microscopic surgeries), we can rescale by $H(X_j)$ instead of $\lambda_j$ to conclude the proof of the theorem.
\end{proof}

\begin{theorem}[existence of mean curvature flow with surgery]
Given any smooth compact mean-convex domain $K_0\subset\mathbb{R}^3$, for suitable choice of surgery parameters $\delta$ and $\mathcal{H}$, there exists a $(\delta,\mathcal{H})$-flow $\{K_t\}_{t\geq 0}$  with initial condition $K_0$.
\end{theorem}

\begin{proof}
Having established Theorem \ref{thm_can_nbd} (canonical neighborhoods), we can now prove existence via a continuity argument similarly as in \cite[Section 4.2]{HK}. 
To this end, fixing $\bar{\delta}>0$ small enough, suppose towards a contradiction that there is a sequence $\K^j$ of $(\delta,\mathcal{H}_j)$-flows with $\delta\leq \bar{\delta}$ and $\mathcal{H}^j\to\infty$, that can only be defined on a maximal time interval $[0,T_j]$ for some
$T_j<\infty$.
Then, it must be the case that we cannot find a minimal collection
of strong $\delta$-necks in $K^j_{T_j}$
as required in Definition \ref{def_MCF_surgery} (flow with surgery), since otherwise we could perform surgeries along an `innermost' such collection of centers $p$, i.e. one for which
\begin{equation}
\sum_p \textrm{dist}(p,\{H=H_{\textrm{trig}}^j\})
\end{equation}
is minimal,  and run smooth mean curvature flow for a short time, contradicting the maximality of $T_j$.
Therefore our goal is to  
produce a minimal separating collection of strong $\delta$-necks for large $j$, to obtain a contradiction.

Let $\mathcal{I}_j$ be the set of points $p\in \partial K^j_{T_j}$ with $H(p)>H^j_{\textrm{neck}}$, and let $\mathcal{J}_j$ be the set of points $p\in \partial  K^j_{T_j}$ with $H(p)=H^j_{\textrm{neck}}$. Then, similarly as in \cite[Claim 4.6]{HK}
there is a large constant $C<\infty$, such that the union $
V_j=\bigcup_{p\in \mathcal{J}_j} B(p,CH^{-1}(p))$, for $j$ large enough, separates $\{H=H^j_{\textrm{trig}}\}$ from $\{H\leq H^j_{\textrm{th}}\}$ in the domain $K^j_{T_j}$.
Let $\hat{\mathcal{J}}_j\subseteq\mathcal{J}_j$ be a minimal subset
such that the union of balls $\cup_{p\in \hat{\mathcal{J}}_j}B(p,CH^{-1}(p))$ has the separation property.
Then, using Theorem \ref{thm_can_nbd} (canonical neighborhoods) and arguing similarly as in the proof of \cite[Claim 4.7]{HK} we see that
given any $\hat\delta>0$, for $j$ large enough all points in $\hat{\mathcal{J}}_j$ are strong $\hat{\delta}$-neck points.
Hence, for large $j$, every $p\in \hat{\mathcal{J}}_j$ is a strong $\delta$-neck point.  These necks are disjoint for large $j$, since otherwise two intersecting $\delta$-necks would lie in a single $\hat\delta$-neck for $\hat{\delta}\ll\delta$, which is impossible by minimality of $\mathcal{J}_j$.
Thus, we have a minimal collection of disjoint strong $\delta$-necks with the separation property; this gives the desired contradiction and thus proves the theorem.
\end{proof}

\bigskip

\bibliography{surgery}

\bibliographystyle{alpha}

\vspace{10mm}

{\sc Robert Haslhofer, Department of Mathematics, University of Toronto,  40 St George Street, Toronto, ON M5S 2E4, Canada}\\

\emph{E-mail:} roberth@math.toronto.edu

\end{document}